\documentclass[12pt,a4paper,oneside]{amsart}
\usepackage{fullpage}
\usepackage{amssymb,amsfonts}
\usepackage{epsfig}
\usepackage{graphicx}

\usepackage{graphics}
\newtheorem{theorem}{Theorem}[section]
\newtheorem{lemma}[theorem]{Lemma}

\theoremstyle{definition}
\newtheorem{definition}[theorem]{Definition}

\numberwithin{equation}{section}

\DeclareMathOperator{\diam}{diam} 

\newcommand{\be}{\begin{equation}}
\newcommand{\ee}{\end{equation}}

\newcommand{\dist}{{\operatorname{dist}}}

\DeclareMathOperator{\rad}{rad}
\DeclareMathOperator{\capacity}{cap}

\makeindex

\def\Xint#1{\mathchoice 
 {\XXint\displaystyle\textstyle{#1}}%
{\XXint\textstyle\scriptstyle{#1}}%
{\XXint\scriptstyle\scriptscriptstyle{#1}}%
 {\XXint\scriptscriptstyle\scriptscriptstyle{#1}}%
 \!\int}
\def\XXint#1#2#3{{\setbox0=\hbox{$#1{#2#3}{\int}$}
 \vcenter{\hbox{$#2#3$}}\kern-.5\wd0}}

 \def\dashint{\Xint-}

\begin{document}
\title{Capacities and Hausdorff measures on metric spaces}
\author{Nijjwal Karak}
\address{Department of Mathematics and Statistics, University of Jyv\"askyl\"a, P.O. Box 35, FI-40014, Jyv\"askyl\"a, Finland}
\email{nijjwal.n.karak@jyu.fi}
\author{Pekka Koskela}
\address{Department of Mathematics and Statistics, University of Jyv\"askyl\"a, P.O. Box 35, FI-40014, Jyv\"askyl\"a, Finland}
\email{pekka.j.koskela@jyu.fi}
\thanks{The authors were partially supported by the Academy of Finland grant number 131477}
\begin{abstract}
In this article, we show that in a $Q$-doubling space $(X,d,\mu),$ $Q>1,$ that supports a $Q$-Poincar\'e inequality and satisfies a chain condition, sets of $Q$-capacity zero have generalized Hausdorff $h$-measure zero for $h(t)=\log^{1-Q-\epsilon}(1/t).$
\end{abstract}
\maketitle
\section{Introduction}
The relation between capacities and generalized Hausdorff measures in $\mathbb{R}^n$ and in metric spaces has been studied for many years. In $\mathbb{R}^n,$ it is known that sets of $p$-capacity zero have generalized Hausdorff $h$-measure zero provided that
\begin{equation}\label{relation}
\int_0^1\left(t^{p-n}h(t)\right)^{\frac{1}{p-1}}\,\frac{dt}{t}<\infty,
\end{equation}
for $1<p\leq n,$ see Theorem 7.1 in \cite{KM72} or Theorem 5.1.13 in \cite{AH96}. In particular, the Hausdorff dimension of such sets does not exceed $n-p.$ Similar results for weighted capacities and Hausdorff measures in $\mathbb{R}^n$ can be found e.g. in \cite{HKM06}.\\
\indent Let us consider a doubling metric space $(X,d,\mu).$ Then a simple iteration argument shows that there is an exponent $Q>0$ and a constant $C\geq 1$ so that
\begin{equation}\label{1}
\left(\frac{s}{r}\right)^Q\leq C\frac{\mu(B(x,s))}{\mu(B(a,r))}
\end{equation}
holds whenever $a\in X$, $x\in B(a,r)$ and $0<s\leq r.$ We say that $(X,d,\mu)$ is $Q$-\textit{doubling} if $(X,d,\mu)$ is a doubling metric measure space and \eqref{1} holds with the given $Q.$ Towards defining our Sobolev space, we recall that a measurable function $g\ge 0$ is an upper gradient of a measurable function $u$ provided
\begin{equation}
\vert u(\gamma(a))-u(\gamma(b))\vert\le\int_{\gamma}g\, ds
\end{equation}
for every rectifiable curve $\gamma : [a,b]\rightarrow X$ \cite{HK98}, \cite{KM98}. We define $W^{1,p}(X),$ $1\le p<\infty,$ to be the collection of all $u\in L^p(X)$ that have an upper gradient that also belongs to $L^p(X),$ see \cite{Sha00}. In order to obtain lower bounds for the capacity associated to  $W^{1,p}(X),$ it suffices to assume a suitable Poincar\'e inequality. We say that $(X,d,\mu)$ supports a $p$-Poincar\'e inequality if there exist constants $C$ and $\lambda$ such that
\begin{equation}\label{PI}
\dashint_B\vert u-u_B\vert\, d\mu\le C\diam(B)\left(\dashint_{\lambda B}g^p\, d\mu\right)^{1/p}
\end{equation}
for every open ball $B$ in $X$, for every function $u : X\rightarrow\mathbb{R}$ that is integrable on balls, and for every upper gradient $g$ of $u$ in $X.$ For simplicity, we will from now on only consider the case of a $Q$-doubling space and we will assume that $p=Q.$\\
\indent In this paper, we study the relation between $Q$-capacity and generalized Hausdorff $h$-measure for $h(t)=\log^{1-Q-\epsilon}(1/t)$ (see Section 2 for the definitions of capacity and Hausdorff $h$-measure) on a $Q$-doubling metric measure space that supports a $Q$-Poincar\'e inequality. Bj\"orn and Onninen proved in \cite{BO05} that a compact set $K$ in a $Q$-doubling space that supports a $1$-Poincar\'e inequality has Hausdorff $h$-measure zero provided that $Q$-capacity of $K$ is zero, for any $h$ that satisfies \eqref{relation} with $n$ replaced by $Q.$ Hence this holds for $h(t)=\log^{1-Q-\epsilon}(1/t)$ for any $\epsilon>0.$ Under the weaker assumption of a $Q$-Poincar\'e inequality, their work shows that $K$ has Hausdorff $h$-measure zero, for $h(t)=\log^{-Q-\epsilon}(1/t).$ They pose an open problem that in our setting asks if the above analogue of \eqref{relation} is sufficient for $h$ even under a $Q$-Poincar\'e inequality assumption. An examination of the corresponding proof in \cite{BO05} shows that it actually suffices that the Poincar\'e inequality \eqref{PI} holds for each $u\in W^{1,Q}(X)$ with $p=1$ for some function $g\in L^Q(X),$ whose $Q$-norm is at most a fixed constant times the infimum of $Q$-norms of all upper gradients of $u.$ This requirement holds for complete $Q$-doubling spaces that supports a $Q$-Poincar\'e inequality by the self-improving property of Poincar\'e inequalities \cite{KZ08}, for details see Section 4 of \cite{KK}. However, the self-improving property from \cite{KZ08} may fail in the non-complete setting, see \cite{Kos99}.\\
\indent We establish the optimal result for logarithmic gauge functions $h$ under a mild additional assumption.
\begin{theorem}\label{maintheorem}
Let $\epsilon>0.$ Let $(X,\mu)$ be a $Q$-doubling space for some $Q>1$ that supports a $Q$-Poincar\'e inequality and assume that $X$ satisfies a chain condition (see definition \ref{chain}). Let $x_0\in X$ and $R>0.$ Then we have $H^h(E)=0$ for every compact $E\subset B(x_0,R)$ with $\capacity_Q(E,B(x_0,2R))=0$, where $h(t)=\log^{1-Q-\epsilon}(1/t).$
\end{theorem}
A doubling space that supports a $p$-Poincar\'e inequality is necessarily connected and even bi-Lipschitz equivalent to a geodesic space, if it is complete \cite{Che99}. Since each geodesic space satisfies a chain condition, the assumption of chain condition in Theorem \ref{maintheorem} is natural.
\section{Notation and preliminaries}
We assume throughout that $X=(X,d,\mu)$ is a metric measure space equipped with a metric $d$ and a Borel regular outer measure $\mu.$ We call such a $\mu$ a measure. The Borel-regularity of the measure $\mu$ means that all Borel sets are $\mu$-measurable and that for every set $A\subset X$ there is a Borel set $D$ such that $A\subset D$ and $\mu(A)=\mu(D).$\\

We denote open balls in $X$ with a fixed center $x\in X$ and radius $0<r<\infty$ by $$B(x,r)=\{y\in X : d(y,x)<r\}.$$
If $B=B(x,r)$ is a ball, with center and radius understood, and $\lambda>0,$ we write
$$\lambda B=B(x,\lambda r).$$
With small abuse of notation we write $\rad(B)$ for the radius of a ball $B$ and we always have $$\diam(B)\leq 2\rad(B),$$
and the inequality can well be strict.\\

A Borel regular measure $\mu$ on a metric space $(X,d)$ is called a \textit{doubling measure} if every ball in $X$ has positive and finite measure and there exist a constant $C\geq 1$ such that
\begin{equation*}
\mu(B(x,2r))\leq C_{\mu}\,\mu(B(x,r))
\end{equation*}
for each $x\in X$ and $r>0.$ We call a triple $(X,d,\mu)$ a \textit{doubling metric measure space} if $\mu$ is a doubling measure on $X.$\\ 

If $A\subset X$ is a $\mu$-measurable set with finite and positive measure, then the \textit{mean value} of a function $u\in L^1(A)$ over $A$ is
$$u_A=\dashint_A u\, d\mu=\frac{1}{\mu(A)}\int_A u\, d\mu.$$
\\
\indent A metric space is said to be \textit{geodesic} if every pair of points in the space can be joined by a curve whose length is equal to the distance between the points.
\begin{definition}
Let $E\subset B(x_0,R)$ be compact. The $Q$-capacity of $E$ with respect to the ball $B(x_0,2R)$ is
\begin{equation*}
\capacity_Q(E, B(x_0,2R))=\inf\Vert g\Vert_{L^Q(X)}
\end{equation*}
where the infimum is taken over all upper gradients $g$ of all continuous functions $u$ with compact support in $B(x_0,2R)$ and $u\geq 1$ on $E.$
\end{definition}
Let $h:[0,\infty)\rightarrow [0,\infty)$ be a non-decreasing function such that $\lim_{t\rightarrow 0+}h(t)=h(0)=0.$ For $0<\delta\leq\infty,$ and $E\subset X,$ we define \textit{generalized Hausdorff $h$-measure} by setting
\begin{equation*}
H^h(E)=\limsup_{\delta\rightarrow 0}H_{\delta}^{h}(E),
\end{equation*}
where
\begin{equation*}
H_{\delta}^{h}(E)=\inf\sum_i h(\diam(B_i)),
\end{equation*}
where the infimum is taken over all collections of balls $\{B_i\}_{i=1}^{\infty}$ such that $\diam(B_i)\leq\delta$ and $E\subset\bigcup_{i=1}^{\infty} B_i.$ In particular, if $h(t)=t^\alpha$ with some $\alpha>0,$ then $H^h$ is the usual \textit{$\alpha$-dimensional Hausdorff measure}, denoted also by $H^{\alpha}.$ See \cite{Rog98} for more information on the generalized Hausdorff measure. Recall that the \textit{Hausdorff $h$-content} of a set $E$ in a metric space is the number
$$\mathcal{H}_{\infty}^h(E)=\inf\sum_ih(\diam(B_i)),$$
where the infimum is taken over all countable covers of the set $E$ by balls $B_i.$ Thus the $h$-content of $E$ is less than, or equal to, the Hausdorff $h$-measure of $E,$ and it is never infinite for $E$ bounded. However, the $h$-content of set is zero if and only if its Hausdorff $h$-measure is zero.\\
\indent For the convenience of reader we state here a fundamental covering lemma (for a proof see \cite[2.8.4-6]{Fed69} or \cite[Theorem 1.3.1]{Zie89}).
\begin{lemma}[5B-covering lemma]\label{cover}
Every family $\mathcal{F}$ of balls of uniformly bounded diameter in a metric space $X$ contains a pairwise disjoint subfamily $\mathcal{G}$ such that for every $B\in\mathcal{F}$ there exists $B'\in\mathcal{G}$ with $B\cap B'\neq\emptyset$ and $\diam(B)<2\diam(B').$ In particular, we have that
$$\bigcup_{B\in\mathcal{F}}B\subset\bigcup_{B\in\mathcal{G}}5B.$$
\end{lemma}
We mention a technical lemma from \cite{KK} and we give a simple proof here.
\begin{lemma}\label{3}
Suppose $\{a_j\}_{j=0}^{\infty}$ is a sequence of non-negative real numbers such that $\sum_{j\geq 0}a_j<\infty$. Then
\begin{equation*}
\sum_{j\geq 0}\frac{a_j}{\left(\sum_{i\geq j}a_i\right)^{1-\delta}}\leq\frac{1}{\delta}\left(\sum_{j\geq 0}a_j\right)^{\delta}<\infty\quad\text{for any}\quad 0<\delta<1.
\end{equation*}
\end{lemma}
\begin{proof}
Define $$u(t)=\sum\limits_{j\geq 0}a_j\chi_{\left[j, j+1\right)}(t)$$ for $t\geq 0$ and $v(x)=\int_x^{\infty}u(t)\,dt$ for $x\geq 0.$ Then $v$ is a Lipschitz function and
$$v(x)=\sum\limits_{j\geq 0}a_jv_j(x),$$
where
\begin{equation*}
v_j(x)=
    \begin{cases}
    1& \text{if $x<j$},\\
    j+1-x& \text{if $j\leq x<j+1$},\\
    0& \text{if $x\geq j+1$}.
    \end{cases}
\end{equation*}
Then we have the required estimate
\begin{equation*}
\sum\limits_{j\geq 0}\frac{a_j}{\left(\sum_{i\geq j}a_i\right)^{1-\delta}}\leq\int_0^{\infty}\frac{-v^{\prime}(x)dx}{v(x)^{1-\delta}}=\frac{1}{\delta}\left(\sum\limits_{j\geq 0}a_j\right)^{\delta}<\infty.
\end{equation*}
\end{proof}
\section{Proof of Theorem \ref{maintheorem}}  
Before we go into the proof of Theorem \ref{maintheorem}, let us recall a definition of a \textit{chain condition} from \cite{KK}, a version of which is already introduced in \cite{HK00}.
\begin{definition}\label{chain}
We say that a space $X$ satisfies a \textit{chain condition} if for every $\lambda\geq 1$ there are constants $M\geq 1,$ $0<m\leq 1$ such that for each $x\in X$ and all $0<r<\diam(X)/8$ there is a sequence of balls $B_0,B_1,B_2,\ldots$ with\\
1. $B_0\subset X\setminus B(x,r)$,\\
2. $M^{-1}\diam(B_i)\leq \dist(x,B_i)\leq M\diam(B_i)$,\\
3. $\dist(x,B_i)\leq Mr2^{-mi}$,\\
4. there is a ball $D_i\subset B_i\cap B_{i+1}$, such that $B_i\cup B_{i+1}\subset MD_i$,\\
for all $i\in\mathbb{N}\cup\{0\}$ and\\
5. no point of $X$ belongs to more than $M$ balls $\lambda B_i$.
\end{definition}
\indent The sequence $B_i$ will be called a \textit{chain associated with} $x,r$.\\

The existence of a doubling measure on $X$ does not guarantee a chain condition. In fact, such a space can be badly disconnected, whereas a space with a chain condition cannot have \lq\lq large gaps\rq\rq . For example, the standard $1/3$-Cantor set satisfies a chain condition only for $\lambda<2$. On the other hand, geodesic and many other spaces satisfy our chain condition, see \cite{KK}.
\indent We recall a lemma from \cite{KK} and we omit the proof here.
\begin{lemma}\label{comparision}
Suppose that $X$ satisfies a \textit{chain condition} and let the sequence $B_i$ be a \textit{chain associated with} $x,R_1,R_2$ for $x\in X$ and $0<R_1<R_2<\diam(X)/4$. Then we can find balls $B_{i_{R_2}},B_{i_{R_2}+1},\ldots,B_{i_{R_1}}$ from the above collection such that
\begin{eqnarray}
\frac{R_2}{M(1+M)^2} & \leq \diam(B_{i_{R_2}}) & \leq MR_2,\\
\frac{R_1}{M(1+M)^2} & \leq \diam(B_{i_{R_1}}) & \leq MR_1
\end{eqnarray}
hold and $B_{i_{R_2}} \subset B(x,R_2),$ $B_{i_{R_1}}\subset B(x,R_1)$ and also the balls $B_{i_{R_2}},B_{i_{R_2}+1},\ldots,B_{i_{R_1}}$ form a chain.
\end{lemma}

\begin{proof}[\textbf{Proof of Theorem \ref{maintheorem}}] For notational simplicity, we assume $R=1/8.$ Let $u$ be a continuous function with compact support in $B(x_0,1/4)$ and $u\geq 1$ on $E.$ Let $g$ be an upper gradient of $u.$ We construct
\begin{equation}\label{estimate}
E_{\epsilon,M}=\left\{x\in E : \exists ~\text{some}~ r_x<10 ~\text{so that}~ \int_{B(x,r_x)}g^Q\,d\mu\geq M\log^{1-Q-\epsilon}\left(\frac{10}{r_x}\right)\right\},
\end{equation}
$M$ to be chosen later.\\
\indent Let $x\in E\setminus E_{\epsilon,M}.$ 
Let $k\in\mathbb{N}.$ Then we apply Lemma \ref{comparision} for $R_1=2^{-k},$ $R_2=2^{-1}$ to get a chain of balls $B_1,B_2,\ldots,B_{i_k}.$ Using the doubling property, Poincar\'e inequality and Lemma \ref{comparision}, we obtain
\begin{eqnarray*}
\vert u_{B_{i_k}}-u_{B(x,2^{-k})}\vert &\leq & \dashint_{B_{i_k}}\vert u-u_{B(x,2^{-k})}\vert \,d\mu\\
&\leq & c\dashint_{u_{B(x,2^{-k})}}\vert u-u_{B(x,2^{-k})}\vert \,d\mu\\
&\leq & c\left(\int_{B(x,2^{-k})}g^Q\,d\mu\right)^{\frac{1}{Q}}\rightarrow 0 ~\text{as}~ k\rightarrow\infty
\end{eqnarray*}
and hence $u_{B_{i_k}}\geq 2/3$ for large $k,$ by the continuity of $u.$ We assume that $u_{B_1}\leq 1/3,$ as we can always do it by increasing the radius $R_2.$\\
Let $\tilde{\epsilon}>0,$ which is to be chosen later.  We use a telescopic argument for the balls $B_1,B_2,\ldots,B_{i_k}$ and also use chain conditions, relative lower volume decay \eqref{1} and Poincar\'e inequality \eqref{PI} to obtain
\begin{eqnarray*}
\frac{1}{3}\leq\vert u_{B_{i_k}}-u_{B_1}\vert & \leq & \sum_{n=1}^{i_k-1}\vert u_{B_n}-u_{B_{n+1}}\vert\\
&\leq & \sum_{n=1}^{i_k-1}\left(\vert u_{B_n}-u_{D_n}\vert+\vert u_{B_{n+1}}-u_{D_n}\vert\right)\\
&\leq & \sum_{n=1}^{i_k}\left(\dashint_{D_n}\vert u-u_{B_n}\vert\, d\mu +\dashint_{D_n}\vert u-u_{B_{n+1}}\vert\, d\mu\right)\\
& \leq & c\sum_{n=1}^{i_k}\dashint_{B_n}\vert u-u_{B_n}\vert\, d\mu\\
& \leq & c\sum_{n=1}^{i_k}\diam(B_n)\left(\dashint_{\lambda B_n}g^Q\, d\mu\right)^{\frac{1}{Q}}\\
& \leq & c\sum_{n\geq 1}\left(\frac{\diam(B_n)^Q}{\mu(B_n)}\int_{\lambda B_n}g^Q\, d\mu\right)^{\frac{1}{Q}}n^{\frac{Q-1+\tilde{\epsilon}}{Q}}n^{-\frac{Q-1+\tilde{\epsilon}}{Q}}\\
& \leq & c\left(\sum_{n\geq 1}\frac{\diam(B_n)^Q}{\mu(B_n)}n^{Q-1+\tilde{\epsilon}}\int_{\lambda B_n}g^Q\, d\mu\right)^{\frac{1}{Q}} \left(\sum_{n\geq 1}n^{-\frac{Q-1+\tilde{\epsilon}}{Q-1}}\right)^{\frac{Q-1}{Q}}\\
& \leq & \frac{c}{\mu(B(x,10))}\left(\sum_{n\geq 1}n^{Q-1+\tilde{\epsilon}}\int_{\lambda B_n}g^Q\, d\mu\right)^{\frac{1}{Q}}.
\end{eqnarray*}
Since $x\in E\setminus E_{\epsilon,M},$ we have
\begin{equation}\label{4}
\int_{B(x,r_x)}g^Q\, d\mu\leq M\log^{1-Q-\epsilon}\left(\frac{10}{r_x}\right)
\end{equation}
for all $r_x<10.$ Hence we get
\begin{equation}
\sum_{m\geq n}\int_{\lambda B_m}g^Q\,d\mu\leq Mn^{1-Q-\epsilon}
\end{equation}
for all $n\geq 1.$ Then we choose $\tilde{\epsilon}=\epsilon-\delta(Q-1-\epsilon)$ for some $0<\delta<1$ (we can choose $\delta$ as small as we want to make $\tilde{\epsilon}$ positive) to obtain
\begin{equation*}
1\leq\frac{cM^{\frac{1-\delta}{Q}}}{\mu(B(x,10))}\left(\sum_{n\geq 1}\frac{\int_{\lambda B_n}g^Q\, d\mu}{\left(\sum_{m\geq n}\int_{\lambda B_m}g^Q\right)^{1-\delta}}\right)^{\frac{1}{Q}}.
\end{equation*}
Finally, we use Lemma \ref{3} and \eqref{4} to get
\begin{eqnarray*}
1 &\leq & \frac{cM^{\frac{1-\delta}{Q}}}{\delta\mu(B(x,10))}\left(\sum_{n\geq 1}\int_{\lambda B_n}g^Q\, d\mu\right)^{\frac{\delta}{Q}}\\
&\leq & \frac{cM^{\frac{1}{Q}}}{\delta\mu(B(x,10))}.
\end{eqnarray*}
If we choose $M<\delta^Q/c^Q,$ where $0<\delta<\epsilon/(Q-1+\epsilon)$, then we get contradiction in the above inequality to conclude that $E\setminus E_{\epsilon,M(\epsilon,Q,c)}=\emptyset.$ In other words, for every $x\in E,$ there exists $r_x<10$ such that
\begin{eqnarray*}
\int_{B(x,r_x)}g^Q\,d\mu\geq M(\epsilon,Q,c)\log^{1-Q-\epsilon}\left(\frac{10}{r_x}\right).
\end{eqnarray*}
By $5B$-covering lemma, pick up a collection of disjoint balls $B_i=B(x_i,r_i)$ such that $E\subset \cup_i 5B_i.$ Then
\begin{eqnarray*}
\int_{X}g^Q\,d\mu\geq\sum_i\int_{B_i}g^Q\,d\mu &\geq & M(\epsilon,Q,c)\sum_i\log^{1-Q-\epsilon}\left(\frac{1}{r_i}\right)\\
&\geq & M(\epsilon,Q,c)\log^{1-Q-\epsilon}\left(\frac{1}{\diam(E)}\right),
\end{eqnarray*}
hence $\capacity_Q(E,B(x_0,2R))\geq M(\epsilon,Q,c)\mathcal{H}^h_{\infty}(E).$
\end{proof}

\def\bibname{References}
\bibliography{capacity}
\bibliographystyle{alpha}
\end{document}